\documentclass{article}
\usepackage[utf8]{inputenc}
\usepackage[utf8]{inputenc}
\usepackage[greek,english]{babel}
\usepackage{alphabeta}
\usepackage{amsfonts}
\usepackage{graphicx}
\usepackage{amssymb}
\usepackage{parskip}
\usepackage{amsthm}
\usepackage{geometry}
\usepackage{array}

\title{A braidoid equivalence for spherical knotoids}
\author{Anastasios Kokkinakis}
\date{}

\theoremstyle{theorem}
\newtheorem{theorem}{Theorem}[section]
\newtheorem{corollary}{Corollary}[theorem]
\newtheorem{proposition}{Proposition}[section]
\newtheorem{lemma}[theorem]{Lemma}
\theoremstyle{definition}
\newtheorem{definition}{Definition}[section]
\newtheorem{remark}{Remark}

\begin{document}
\setlength{\parindent}{20pt}

\maketitle

\begin{abstract}
    Braidoids form a counterpart theory to the theory of planar knotoids, just as braids do for three-dimensional links. As such, planar knotoid diagrams represent the same knotoid in $\mathbb{R}^2$ if and only if they can be presented as the closure of two labeled braidoid diagrams related by an equivalence relation, named $L$-equivalence. In this paper, we refine the notion of $L$-equivalence of braidoid diagrams in order to obtain an equivalence theorem for (multi)-knotoid diagrams in $S^2$ when represented as the closure of labeled braidoid diagrams.
\end{abstract}

\section{Introduction}

A knotoid in an oriented surface $Σ$ is an equivalence class of oriented open-ended curve diagrams in $Σ$, with two endpoints that can lie in any local region determined by the diagram. The equivalence is generated by the Reidemeister moves and isotopies of $Σ$, which include the swinging of an endpoint within the region it lies. The theory of knotoids was first introduced by Turaev \cite{turaev2012vladimir} in 2011 as a way to reduce complexity in the study of knots. Indeed, Turaev observed that an open curve diagram drawn in the sphere, endowed with over/under data on the double points, can carry the same, if not more, information as a knot diagram, and proved that the theory of spherical knotoids extends classical knot theory. Knotoids are also of interest in biology. Indeed, recent studies by Goundaroulis et al. \cite{dabrowski2019theta, goundaroulis2017studies, goundaroulis2017topological} contribute to the topological classification of proteins by analysing the distribution of their projections as spherical or planar knotoids.\\
\indent Since knotoids are defined in a two-dimensional environment as supposed to knots and links, which are objects embedded in three-dimensional space, differences in the theory arise when considering knotoids in different surfaces.  Indeed, due to the spherical move, equivalent knotoids in $S^2$ may not correspond to equivalent knotoids in $\mathbb{R}^2$. Furthermore, as it turns out, the theory of planar knotoids extends the theory of spherical knotoids.\\
\indent Turaev also linked the theory of spherical knotoids with the the theory of $θ$-curves in three-space, proving that spherical knotoids correspond bijectively to label-preserving ambient isotopy classes of simple theta-curves. In analogy, Gügümcü and Kauffman in \cite{gugumcu2017new} proved a bijective correspondence of planar knotoids to ambient isotopy classes of open-ended curves in three-space with their ends attached to two parallel lines. \\
\indent The theory of braidoids was introduced in 2017 by Gügümcü and Lambropoulou \cite{gugumcu2017knotoids} as the counterpart theory to knotoids in analogy to braids being the algebraic counterpart objects to knots/links.  A braidoid diagram consists of a set of descending strands such that two of them are special: one of them terminates at an endpoint, the head, and the other starts from the second endpoint, the leg. Either endpoint may lie in any region and at any height of the diagram.  The study of braidoids was naturally linked with the study of knotoids in $\mathbb{R}^2$ rather than knotoids in $S^2$. In \cite{gugumcu2017knotoids} it is shown that every planar knotoid can be transformed equivalently to the closure of a braidoid. This result corresponds to the classical Alexander theorem. Moreover, in classical knot theory we have the Markov equivalence in the set of braids, which generates ambient isotopy between links in $S^2$ when we consider them as closed braids. In \cite{lambropoulou1997markov} Lambropoulou and Rourke introduced the $L$-moves, a set of moves of a single type which generate the Markov equivalence. The advantage of the $L$-equivalence is that it can be formulated without considering the group structure of braids. Braidoids do not possess an obvious alegbraic structure, so a braidoid analog of the $L$-equivalence was formulated by Gügümcü and Lambropoulou in order to prove the analogous equivalence theorem for knotoids in $\mathbb{R}^2$.\\
\indent Spherical knotoids are of special interest as far as classical knot theory is concerned. With this in mind, our aim is to link the theory of braidoids with the theory of spherical knotoids. To this end, we first show that Turaev's normal spherical knotoids (i.e. spherical knotoids with the leg lying in the connected component of the point at infinity)  correspond bijectively to the so-called planar knotoids in spherical state (see Theorem~\ref{thm:normalbraidoiding}). We then show that every planar knotoid in spherical state can be transformed equivalently to the closure of a so-called braidoid in spherical state. We finally present a set of moves between braidoids,  which realize the spherical move in the closures (see Definitions ~\ref{slI} and \ref{slII}), and which extend the moves in \cite{gugumcu2017knotoids}, such that equivalence classes of braidoids under the augmented set of moves correspond bijectively to spherical knotoids.\\
\indent The paper is organized as follows. In Section~\ref{sec:knotoidsandsimplethetacurves} we recall basics on knotoids and their geometric realization as theta-curves. We discuss equivalence between knotoid diagrams and the differences in the theory of planar vs spherical knotoids. In section \ref{sec:braidoids} we discuss braidoids. We recall the definitions of braidoid isotopy and labeled braidoid equivalence. Finally we recall the definition of the closure of a braidoid diagram, we detail the steps of the braidoing algorithm and of $L$-equivalence between labeled braidoid diagrams. In section \ref{sec:sphericalknotoidsandbraidoidsinsphericalstate} we define the spherical state of a braidoid diagram and the effect of the braidoiding algorithm on normal knotoid diagrams. In \ref{sec:Sphericalequivalenceforlabeledbraidoids} we introduce two sets of moves on labeled braidoid diagrams and give a refinement of the $L$-equivalence. Lastly, we give an analog of the Markov theorem for braidoids and spherical knotoids.\\
\indent The author would like to thank professor Sofia Lambropoulou for her guidance and patience during the completion of this paper.

\section{Knotoids}\label{sec:knotoidsandsimplethetacurves}

In this section we present the basic notions from the theory of knotoids. We focus particularly to the cases of our interest, namely spherical and planar knotoids, which we compare via the stereographic projection. After a series of observations we establish Theorem \ref{thm:sphericalstate} which links spherical knotoids with planar knotoids in spherical state.

\subsection{Planar and spherical knotoids}\label{subsec:planarsphericalknotoids}

A \textit{planar knotoid diagram} $K$ in $\mathbb{R}^2$ is a generic immersion of the interval $[0,1]$ in the interior of $\mathbb{R}^2$ whose only singularities are transversal double points endowed with over/under-crossing data. The images of $0$ and $1$ under this immersion are called the leg and the head of $K$, respectively. These two points are distinct from each other and from the double points; they are called the endpoints of $K$. We orient $K$ from the leg to the head. The double points of $K$ are called the crossings of $K$.\\
\indent Two planar knotoid diagrams are said to be \textit{planar equivalent} if they can be related by a finite sequence of local isotopies of $\mathbb{R}^2$ and applications of the Reidemeister moves $RI,RII,RIII$ away from the endpoints. The equivalence class of a planar knotoid diagram shall be called a \textit{planar knotoid}. The set of planar knotoids shall be denoted $\mathcal{K}(\mathbb{R}^2)$. When $K_1$ and $K_2$ are planar equivalent knotoid diagrams we shall write $K_1 \simeq_{\mathbb{R}^2} K_2$.\\
\indent A \textit{fake forbidden move} on a knotoid diagram $K$ is when an endpoint has seemingly crossed over or under a segment of an arc, however the move can be as a series of equivalence moves and/or planar isotopies (see Fig. \ref{fig:fake}).

\begin{figure}[htp]
    \centering
    \includegraphics[width=7cm]{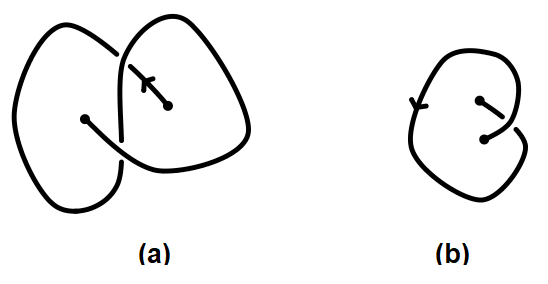}
    \caption{Examples of  knotoid diagrams.}
    \label{fig:knotoid}
\end{figure}

\begin{figure}[htp]
    \centering
    \includegraphics[width=7cm]{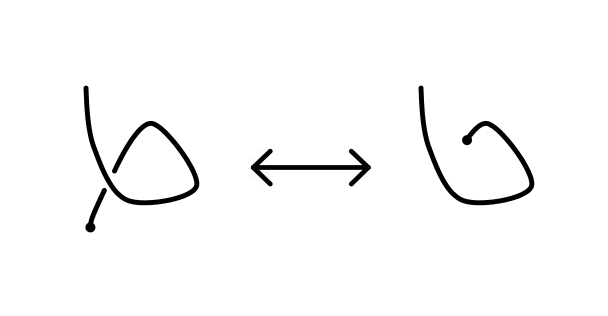}
    \caption{Example of a fake forbidden move on a knotoid diagram.}
    \label{fig:fake}
\end{figure} 

A \textit{spherical knotoid diagram} $K$ is analogously defined as a generic immersion of the interval $[0,1]$ in the interior of $S^2$ whose only singularities are transversal double points endowed with over/under-crossing data.\\
\indent \textit{Multi-knotoid} diagrams in $\mathbb{R}^2$ (res. $S^2$) are defined as generic immersions of a single oriented segment and some oriented circles in $\mathbb{R}^2$ (res. $S^2$) endowed with over/under-crossing data.

\begin{definition}
    A \textit{spherical move on a spherical multi-knotoid diagram} is a global isotopy move on $S^2$ whereby an arc is pushed around the back of the sphere by a finite sequence of local sphere isotopies. Similarly, a \textit{spherical move on a planar multi-knotoid diagram} is a global move whereby an arc in the unbounded region is transferred across the other side of the multi-knotoid diagram. The resulting arc is lying also in the unbounded region, see Fig. \ref{fig:sphericalmove}.  
\end{definition}

Note that, unlike classical knot diagrams, a spherical move on a planar multi-knotoid diagram is not an equivalence move of planar multi-knotoid diagrams, since it necessarily involves the forbidden moves.

\begin{definition}
Two spherical multi-knotoid diagrams are said to be \textit{spherical equivalent} if they can be related by a finite sequence of local isotopies of $S^2$, applications of the Reidemeister moves $RI,RII,RIII$ away from the endpoints, together with the spherical moves. The equivalence class of a spherical multi-knotoid diagram $K$ shall be called \textit{spherical multi-knotoid}, denoted $k$.  The set of spherical knotoids shall be denoted $\mathcal{K}(S^2)$. When $K_1$ and $K_2$ are spherical equivalent multi-knotoid diagrams we shall write $K_1 \simeq_{S^2} K_2$.

\end{definition}

\begin{figure}[htp]
    \centering
    \includegraphics[width=10cm]{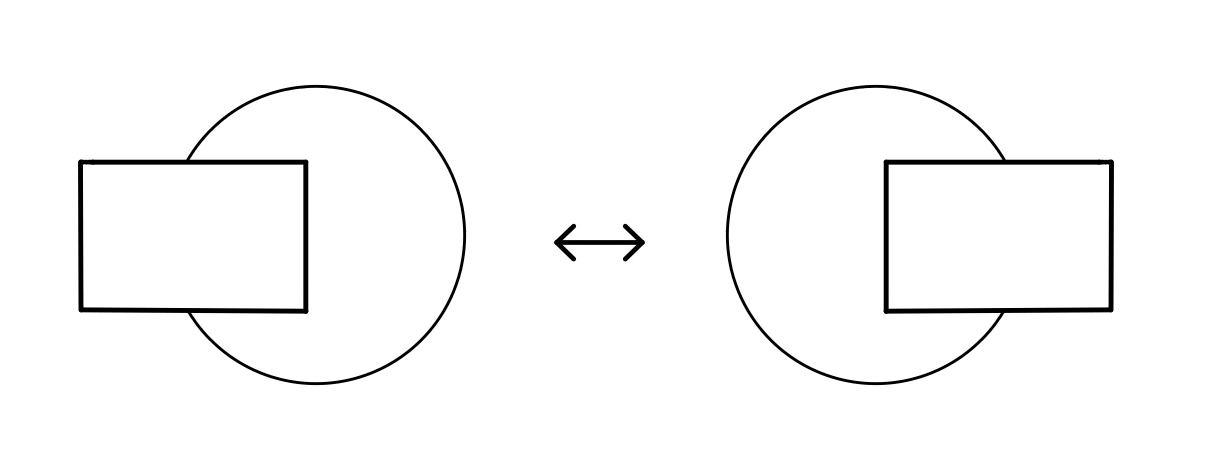}
    \caption{The spherical move.}
    \label{fig:sphericalmove}
\end{figure} 

\textit{Observation 1:} Let us now consider $S^2$ as the compactification of $\mathbb{R}^2$, $S^2 = \mathbb{R}^2 \cup \{\infty\}$. 
 Equivalently, $\mathbb{R}^2$ is homeomorphic with $S^2 \setminus \{\infty\}$, and we shall write $\mathbb{R}^2 = S^2 \setminus \{\infty\}$. The stereographic projection $s$  of $S^2 \setminus \{\infty\}$ on $\mathbb{R}^2$ is a specific homeomorphism and its inverse  $s^{-1}$ can be viewed as an inclusion map of $\mathbb{R}^2$ in $S^2$. 
 
 \textit{Observation 2:} By a general position argument, a knotoid diagram in  $S^2 = \mathbb{R}^2 \cup \{\infty\}$ can be assumed to miss the point at infinity.   The stereographic projection $s$ induces then a well-defined bijective map  from the set of spherical knotoid diagrams to the set of planar knotoid diagrams. Note that, an arc crossing the point at infinity performs precisely a spherical move.

\textit{Observation 3:} Two  planar knotoid diagrams $K_1, K_2$ are planar equivalent (i.e. differ by moves RI, RII, RIII and planar isotopy) if and only if $s^{-1}(K_1), s^{-1}(K_2)$ differ by moves RI, RII, RIII and local isotopies, excluding the spherical moves.
Hence, $s^{-1}$ induces a well-defined map from $\mathcal{K}(\mathbb{R}^2)$ to  $\mathcal{K}(S^2)$, which is surjective but not injective. Indeed, due to the extra spherical moves in $S^2$, different equivalence classes of planar knotoids may map to the same spherical knotoid. 
 
\textit{Observation 4:} A knotoid diagram  $K$ in $S^2$ is an open curve diagram. Letting $K_0$ be the 4-valent graph obtained from $K$ by forgetting the over/under data on the crossings, we have that its complement $S^2 \setminus K_0$ has a component which encloses in its interior the point at infinity. Using  spherical moves, Turaev pointed out that any spherical multi-knotoid $k$ has a (at least one) representative multi-knotoid diagram $N$ in $S^2$ where one endpoint (preferably the leg) and the point at infinity are interior points of the same component of $S^2 \setminus K_0$.  Such a diagram is said to be \textit{normal}. 

 \begin{definition}
 The equivalence relation in the set of  multi-knotoid diagrams in $S^2$, generated by moves RI, RII, RIII and local isotopies, excluding the spherical moves, shall be referred to as \textit{normal equivalence}. 
\end{definition}

\noindent Restricting now to normal spherical diagrams we have:

\begin{proposition}[Cf.  \cite{turaev2012vladimir}] \label{normal correspondence}
Τwo  multi-knotoid diagrams $K_1,K_2$  in $S^2$ are spherical equivalent if and only if any two normal representatives of theirs $N_1, N_2$  in $S^2$ are normal equivalent. That is, the elements of $\mathcal{K}(S^2)$ are in bijective correspondence with equivalence classes of normal multi-knotoid diagrams under normal equivalence.
\end{proposition}

\textit{Observation 5:} Two equivalent multi-knotoid diagrams in $\mathbb{R}^2$ have inverse images via $s^{-1}$ that are also equivalent in $S^2$. Similarly, the  images $s(K_1), s(K_2)$  of two normal equivalent diagrams $K_1,K_2$ through the stereographic projection are related by the moves RI, RII, RIII in $\mathbb{R}^2$ and planar isotopy, so they represent the same multi-knotoid in $\mathbb{R}^2$. 

\begin{remark}\label{rem:breakingclasses}
    Hence, the equivalence class of a spherical knotoid will, in general, break into a multitude of normal equivalence classes, any two of which differ by spherical moves, and which correspond bijectively to equivalence classes of planar knotoids. 
\end{remark}

\noindent For an example of the differences in the theory between $S^2$ and $\mathbb{R}^2$, the knotoids corresponding to the diagrams in Fig. \ref{fig:knotoid} are both the trivial knotoid in $S^2$, where in $\mathbb{R}^2$ Turaev \cite{turaev2012vladimir} has shown them to be non-trivial.
 
\textit{Observation 6:} From the above,  when the stereographic projection $s$ is restricted to normal multi-knotoid diagrams,  we have a well-defined injective map from normal equivalence classes of normal multi-knotoid diagrams in $S^2$  to planar multi-knotoids.
 
\begin{definition} \label{def:sphericalstate}
    The image of a normal spherical multi-knotoid diagram via the stereographic projection $s$ is a planar multi-knotoid diagram with its leg lying in the unbounded region of $\mathbb{R}^2 \setminus K_0$, and shall be also referred to as being  \textit{in spherical state}. 
\end{definition}  

\noindent The above observations leads to the following:
\begin{theorem} \label{thm:sphericalstate}
Τhe elements of $\mathcal{K}(S^2)$ are in bijective correspondence with  elements of $\mathcal{K}({\mathbb R}^2)$ in spherical state.     
\end{theorem}

\noindent In the present text we will focus on spherical multi-knotoid diagrams.

\subsection{Simple theta-curves in $S^3$ and a lifting of planar knotoids}

A \textit{theta-curve} $θ$ is a graph embedded in $S^3$ and formed by two vertices $v_0, v_1$ and three edges $e_{-}, e_0, e_{+}$ each of which joins $v_0$ to $v_1$. We call $v_0$ and $v_1$ the leg and the head of $θ$ respectively. Each vertex $v \in {v_0,v_1}$ of $θ$ has a closed $3$-disk neighborhood $B \subset S^3$ meeting $θ$ along precisely $3$ radii of $B$. We call such $B$ a \textit{regular neighborhood of $v$}. The sets $θ_{-}=e_0 \cup e_{-}$$ $$θ_{0}=e_{=} \cup e_{-},θ_{+}=e_0 \cup e_{+}$ are called \textit{constituent knots of} $θ$ and are oriented from $v_0$ to $v_1$ on the edges $e_{0},e_{-},e_{+}$ respectively.\\
\noindent By \textit{isotopy of theta-curves}, we mean ambient isotopy in $S^3$ preserving the labels $0, 1$ of the vertices and the labels ${-}, 0, +$ of the edges. The set of isotopy classes of theta-curves will be denoted $Θ$.

A \textit{simple theta-curve} $θ$ is a theta-curve where $θ_{0}=e_{=} \cup e_{-}$ is the unknot. A theta-curve isotopic to a simple theta-curve is itself simple. In \cite{turaev2012vladimir}, Turaev proved that knotoid diagrams in $S^2$ are in bijection with \textit{simple theta-curves} in $S^3$. Therefore a geometric realization for a knotoid diagram is the middle arc on a simple theta-curve in $S^3$. Furthermore, in \cite{turaev2012vladimir} a semigroup structure is defined for $\mathcal{K}(S^2)$
 and for simple theta-curves in $S^3$.

\begin{theorem}[Turaev]\label{thm:theta}
\label{Tur1}
There is a semigroup isomorphism between spherical knotoids and label-preserving ambient
isotopy classes of simple theta-curves.
\end{theorem}

A \textit{Multi-theta curve with $n$ components}  is an embedding in $S^3$ of the set $(S^1)^{n-1} \times Θ$, where $Θ$ is the standard theta-graph consisting of three edges and two vertices. A multi theta-curve with $n$ components consists of a link of $n-1$ knots and a theta-curve. A multi theta-curve is called \textit{simple} when the theta-curve component is simple.

In analogy to Theorem~\ref{thm:theta}, Gügümcü and Kauffman proved in \cite{gugumcu2017new}  that there is a bijective correspondence between planar knotoids and ambient isotopy classes of open-ended curves in three-space with their ends attached to two parallel lines. This representation  was subsequently used in \cite{goundaroulis2017topological} for the topological study of proteins via planar knotoids.

\section{Braidoids and braidoid $L$-equivalence}\label{sec:braidoids}

The theory of braidoids is introduced in \cite{gugumcu2021braidoids} as the braid counterpart theory of planar knotoids. We shall recall the definition of a braidoid, braidoid  isotopy, labelled braidoids and braidoid closure, and an algorithm for turning a planar knotoid into a (labelled) braidoid with equivalent closure. We finally recall the $L$-moves between braidoids and the $L$-equivalence in the set of braidoids, which generates a bijection with equivalence classes of planar knotoids.

\subsection{Braidoid diagrams and isotopy moves}

We identify $\mathbb{R}^2$ with the $xy$-plane with the $y$-axis directed downward.

\begin{definition}
    A \textit{braidoid diagram} $b$ is a system of a finite number of arcs immersed in $[0,1] \times [0,1] \subset \mathbb{R}^2$.  The arcs of $b$ are called the strands of $b$. Following the natural orientation of $[0,1]$, each strand is naturally oriented downward, with no local maxima or minima. There are only finitely many intersection points among the strands, which are transversal double points endowed with over/under data. Such intersection points are called crossings of $b$. A braidoid diagram has two types of strands, the classical strands and the free strands. A classical strand is like a braid strand connecting a point in $[0,1] \times \{0\}$ to a point  $[0,1] \times \{1\} $. A free strand either connects a point in $[0,1] \times \{0\}$ or $[0,1] \times \{1\}$ to an endpoint that is located anywhere in $[0,1] \times [0,1]$ or connects two endpoints located anywhere in $[0,1] \times [0,1]$. The endpoints are two special points that specifically named as the \textit{leg} and the \textit{head} and are emphasized by graphical nodes labeled by $l$ and $h$, respectively, in analogy with the endpoints of a knotoid diagram. The head is the endpoint that is terminal for a free strand with respect to the orientation, while the leg is the starting endpoint for a free strand with respect to the orientation.\\
    \indent The ends of the strands of $b$ other than the endpoints are called \textit{braidoid ends}. We assume that braidoid ends lie equidistantly on the top and the bottom lines and none of them is vertically aligned with any of the endpoints. It is clear that the number of braidoid ends that lie on the top line is equal to the number of braidoid ends that lie on the bottom line of the diagram. The braidoid ends on top and bottom lines are arranged in pairs so that they are vertically aligned, and are called { \it corresponding ends}. For an example of a braidoid diagram see Fig.~\ref{fig:vertical}.
\end{definition}

{\bf Moves on braidoid diagrams:}  Like for knotoid diagrams, we forbid to pull/push an endpoint of a braidoid diagram over or under a strand. These are the \textit{forbidden moves} on braidoid diagrams, as it is clear that allowing forbidden moves can cancel any braiding of the free strands. There are two types of local moves generating the \textit{braidoid isotopy} on braidoid diagrams, movements on segments of strands and movements of the endpoints. More precisely:

\begin{enumerate}
    \item \textit{$Δ$-moves.} A \textit{braidoid $Δ$-move} replaces a segment of a strand with two segments forming a triangular disk, which has no intersection with endpoints and other arcs except the one edge before and the two edges after, whilst the downward orientation of the strands is preserved (see Fig.~\ref{fig:Δ}). Note that Reidemeister moves RII and RIII away from endpoints on braidoid diagrams are special cases of braidoid $Δ$-moves.
    \item \textit{Vertical moves}. As shown in Fig.~\ref{fig:vertical}, the endpoints of a braidoid diagram can be pulled up or down in the vertical direction as long as they do not violate any of the forbidden moves. Such moves are called \textit{vertical moves}.
    \item \textit{Swing Moves}. An endpoint can also swing to the right or the left like a pendulum (see Fig. \ref{fig:swing}) as long as the downward orientation on the moving arc is preserved, and the forbidden moves are not violated.
\end{enumerate}

\begin{figure}[htp]
    \centering
    \includegraphics[width=7cm]{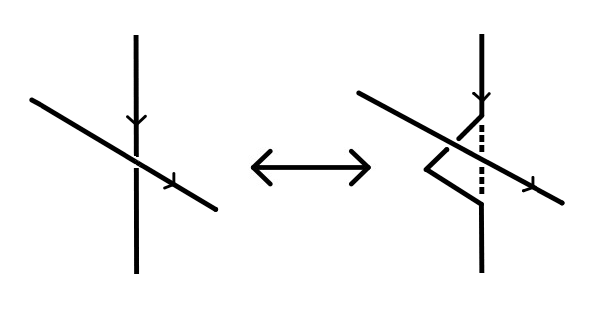}
        \caption{A planar $Δ$-move on a braidoid diagram.}
    \label{fig:Δ}
\end{figure}

\begin{figure}[htp]
    \centering
    \includegraphics[width=6cm]{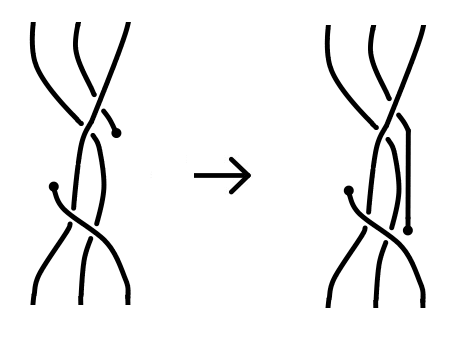}
        \caption{A vertical move on a braidoid diagram.}
    \label{fig:vertical}
\end{figure}

\begin{figure}[htp]
    \centering
    \includegraphics[width=7cm]{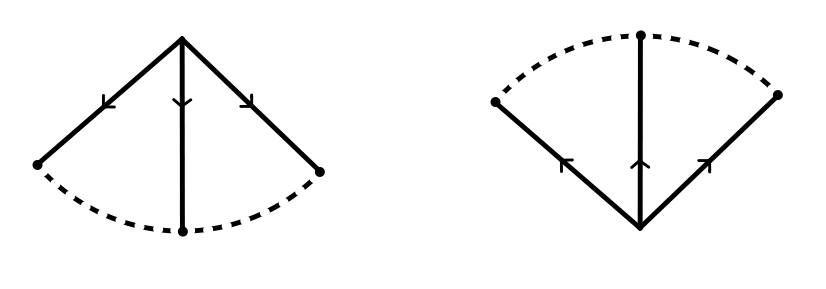}
        \caption{The swing moves.}
    \label{fig:swing}
\end{figure}

\begin{definition}
Two braidoid diagrams are said to be  \textit{isotopic} if one can be obtained from the other by a finite sequence of braidoid isotopy moves. An isotopy class of braidoid diagrams is called a  \textit{braidoid}. 
\end{definition}

\noindent For more details on braidoids cf. \cite{gugumcu2021braidoids}.

\subsection{Labeled braidoid diagrams, closure and restricted braidoid isotopy}

In order to define the closure of a braidoid diagram for obtaining a multi-knotoid we need to introduce the notion of a labeled braidoid diagram.

\begin{definition}
A \textit{labeled braidoid diagram} is a braidoid diagram for which every pair of corresponding ends is labeled either with $o$ or $u$. Each label indicates an over-passing or under-passing arc, respectively.
\end{definition}

\noindent We are now in a position to define the closure of a braidoid diagram. 

\begin{definition}[cf. \cite{gugumcu2021braidoids}]
    The \textit{closure} of  a labeled braidoid diagram $b$ is a planar (multi)-knotoid diagram obtained by connecting each pair of corresponding ends of $b$ with an embedded arc (with slightly tilted extremes) that runs along the right hand-side of the vertical line passing through the ends and in a distance arbitrarily close to this line. The connecting arc goes entirely {\it over} or entirely {\it under }  the rest of the diagram, according to the label of the ends. 
\end{definition}

\begin{figure}[htp]
    \centering
    \includegraphics[width=8cm]{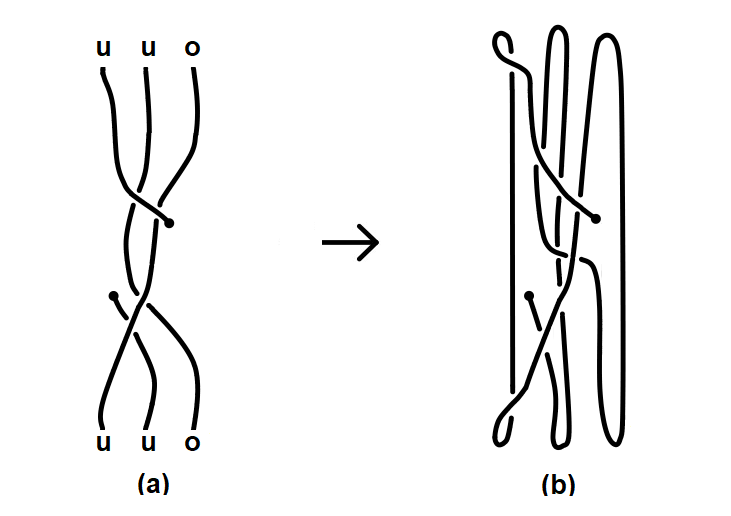}
        \caption{A labeled braidoid diagram whose closure is equivalent to the knotoid in Fig.~\ref{fig:knotoid}a.}
    \label{fig:closure}
\end{figure}

Labeled braidoid diagram isotopy differs slightly from usual braidoid diagram isotopy in that here the vertical lines between corresponding upper and lower ends play a more significant role when  considering their closures later on. In labeled braidoid isotopy we have {\it restricted swing moves} in the following sense: the range of a swing move is restricted, so that the endpoint will not swing outside the vertical strip determined by two closure arcs joining two consecutive pairs of corresponding braidoid ends. An endpoint crossing the imaginary line connecting corresponding ends can result in non-equivalent planar multi-knotoids, since a forbidden move will take place in the closure. As an example, the reader is referred to Fig. \ref{fig:sphernonspher}, where  two seemingly isotopic braidoids which are not restricted isotopic, give rise upon closure  to two non-equivalent planar knotoids. The non-equivalence is established by taking in both knotoids the under-closures, that is joining the endpoints with simple under-arcs    (note that the over-closures give in both cases the trivial knotoid).

\begin{figure}[htp]
    \centering
    \includegraphics[width=14cm]{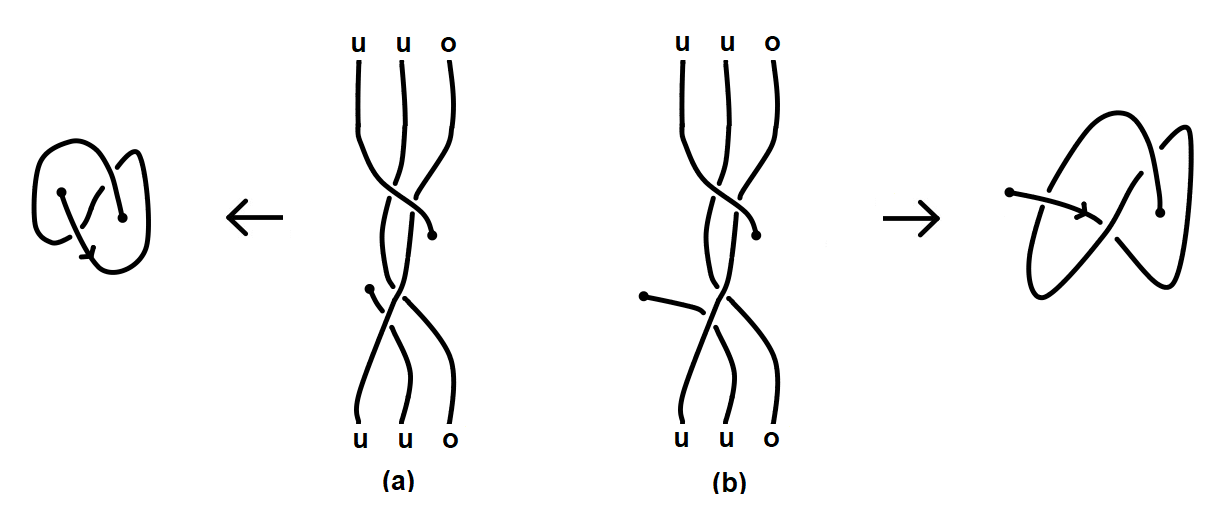}
        \caption{An example of non-isotopic braidoids in terms of restricted isotopy.}
    \label{fig:sphernonspher}
\end{figure}

\begin{definition}
    In the set of labeled braidoid diagrams, the braidoid $Δ$-moves (which include the braidoid Reidemeister moves) along with the vertical moves and the restricted swing moves, preserving at the same time the labeling, generate the \textit{restricted braidoid isotopy}. Isotopy classes of labeled braidoid diagrams under restricted braidoid isotopy are called \textit{labeled braidoids}.
\end{definition}

Upon closure, two restricted isotopic labeled braidoid diagrams generate equivalent multi-knotoid diagrams in $\mathbb{R}^2$, since due to the restricted swing moves no forbidden moves are caused. Hence, the closure operation is well-defined on labeled braidoids.

\subsection{Braidoiding knotoids}\label{subsec:braidoiding}

There exists an inverse operation to  labeled braidoid closure, which is an analogue of the classical Alexander theorem for braiding knot and link diagrams. Let $K$ be a piecewise-linear planar multi-knotoid diagram, endowed with its natural orientation, so that no edges run parallel to the $x$-axis. One can successively eliminate edges that are oriented upward by replacing each one with two new strands with corresponding ends, such that if we join the ends with a closure arc of the same label, the resulting multi-knotoid is equivalent to the original  multi-knotoid.\\
\indent More precisely, let $QP$ be an up-arc edge of a (piecewise linear) multi-knotoid diagram with respect to a given subdivision, where $Q$ denotes the initial subdivision point and $P$  the top-most  one. The right angled triangle labeled $o$ or $u$ according to the label of the up-arc itself, which lies below $QP$ and admits $QP$ as hypotenuse, and is a special case of a triangle enabling a $Δ$-move, is called the \textit{sliding triangle} of the up-arc $QP$. We denote the sliding triangle by $T(P)$, see Fig. \ref{fig:triangle}. The disk bounded by the sliding triangle $T(P)$ is utilized after cutting $QP$, for sliding downward the resulting lower sub-arc across it. See Fig. \ref{fig:triangle}. Α \textit{cut-point} of an up-arc is defined to be the point where the up-arc is cut to start a braidoiding move. We pick the top-most point $P \in QP$ as the cut-point of $QP$ for the braidoiding algorithm. A sliding triangle of a knotoid diagram is not allowed to contain an endpoint. By \cite{gugumcu2021braidoids}, given a point  subdivision of a piecewise linear multi-knotoid, we can always obtain a refinement such that: each up-arc is of the type \textit{under} or \textit{over}, depending on whether it crosses entirely over or entirely under other arcs of the diagram. Moreover, all sliding triangles are mutually disjoint and not containing endpoints. These last requirements are called the \textit{triangle conditions}.\\
\indent Figure \ref{fig:triangle} describes the process of a braidoiding move on an up-arc $QP$, labeled \textit{under}. The sliding triangle's label is then also $u$ and it means that the triangle lies entirely under any other edge of the diagram besides $QP$. Then we join $P$ with its projection on $\mathbb{R} \times \{1\}$ by a line segment that, in this case, passes entirely under any other edge of the diagram. We perform the exact analogous action on $A$ and its projection on $\mathbb{R} \times \{0\}$ and we connect $Q$ and and $A$ with a line segment $QA$ that also passes under any other edge of the diagram. We give the projections of $P$ and $A$ on $\mathbb{R} \times \{1\}$ and $\mathbb{R} \times \{0\}$ indication $u$ for when the braidoiding algorithm is finished these projections will function as braidoid ends labeled $u$. To end the move, we can also slide $A$ $ε$-close downwards since braidoid diagram strands need to be monotonically decreasing.\\
\indent When all up-arcs of $K$ are eliminated in the aforementioned way, the result will be a labeled braidoid diagram whose closure is planar equivalent to $K$. For more details cf. \cite{gugumcu2021braidoids} and \cite{lambropoulou1997markov}.  Namely we have the following theorem due to Gügümcü and Lambropoulou:

\begin{figure}[htp]
    \centering
    \includegraphics[width=13cm]{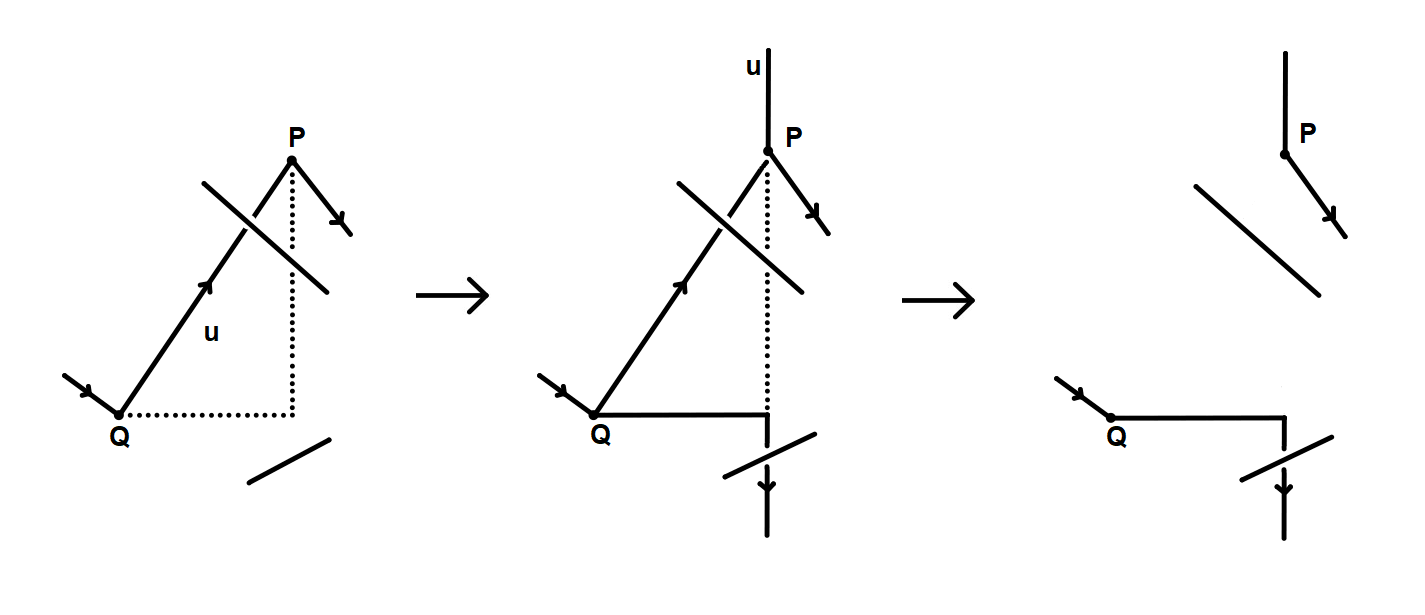}
        \caption{The sliding triangle of the $u$-labeled up-arc $QP$ with $P$ the cut point and the elimination of $QP$.}
    \label{fig:triangle}
\end{figure}

\begin{theorem}
\label{Alexander}
    A  multi-knotoid diagram in $\mathbb{R}^2$ can be turned into a labeled braidoid diagram whose closure is an equivalent multi-knotoid diagram.
\end{theorem}

\noindent As an example, the left-hand side of Fig.~\ref{fig:L1} illustrates a braidoid diagram of the knotoid in Fig.~\ref{fig:knotoid}a.

\begin{figure}[htp]
    \centering
    \includegraphics[width=8cm]{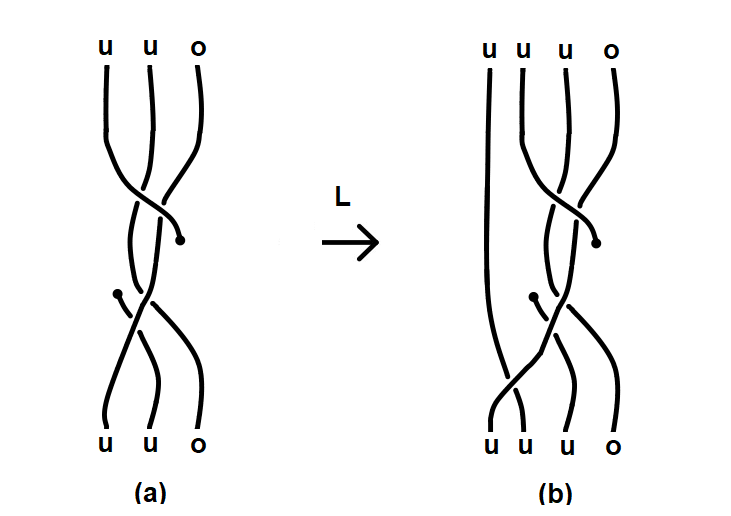}
        \caption{Two equivalent braidoid diagrams closing to the knotoid in Fig.~\ref{fig:knotoid}a.}
    \label{fig:L1}
\end{figure}

A sliding triangle's label can be influenced whether or not the up-arc in question contains a crossing with another strand of the multi-knotoid. However, with small planar isotopies it was shown in \cite{gugumcu2017knotoids} that when an up-arc contains a crossing it can be turned into a down-arc so that the braidoiding algorithm ignores it. Thus, all up-arcs are free and can be given any label. Chosing the same labeling $u$ for all up-arcs we obtain a uniform under labeling for the resulting braidoid diagram. Closing a braidoid diagram with arcs that pass entirely under the diagram is called the \textit{uniform closure} of the diagram.

\subsection{The $L$-equivalence on braidoids}

The $L$-moves were originally defined for classical braid diagrams in \cite{lambropoulou1997markov} by Lambropoulou, and they were used for proving a one-move analogue of the classical two-move Markov theorem \cite{lambropoulou1997markov}. In \cite{gugumcu2021braidoids} the $L$-moves are adapted to the category of braidoids and, in lack of algebraic structures for braidoids, are used for formulating and proving an   $L$-equivalence relation in the set of braidoids that corresponds to equivalence of planar knotoids. The braidoid   $L$-equivalence turns out to be far more subtle than the classical   $L$-equivalence, due to the presence of the fake swing moves and the forbidden moves in the theory.

\begin{definition}[L-moves]
    An $L$-move on a labeled braidoid diagram $b$ is the following operation:

    \begin{enumerate}
  \item Cut a strand of $b$ at an interior point which is not vertically aligned with a braidoid end, an endpoint or a crossing of $b$. The existence of such point can be ensured by a general position argument.
  \item Pull the resulting ends away from the cut-point toward the top and bottom respectively, keeping them vertically aligned with the cut-point, so as to create a new pair of braidoid strands with corresponding ends. The new sub-strands run both entirely \textit{over} or entirely \textit{under} the rest of the braidoid diagram, depending on the type of the $L$-move applied, namely $L_{over}$ or $L_{under}$-move. The two types are denoted by $L_o$ (see Fig. \ref{fig:L-moves}b) and $L_u$ (see Fig. \ref{fig:L-moves}a) respectively. So, an $L_o$-move resp. $L_u$-move comprises pulling the resulting sub-strands entirely over resp. under the rest of the diagram. 
  \item After an $L$-move is applied on a labeled braidoid diagram, the new pair of corresponding strands gets the labeling of the type of the $L$-move: If the strands are obtained by an $L_o$-move resp. $L_u$-move then they are labeled $o$ resp. $u$.
\end{enumerate}
\end{definition}

\begin{figure}[htp]
    \centering
    \includegraphics[width=13cm]{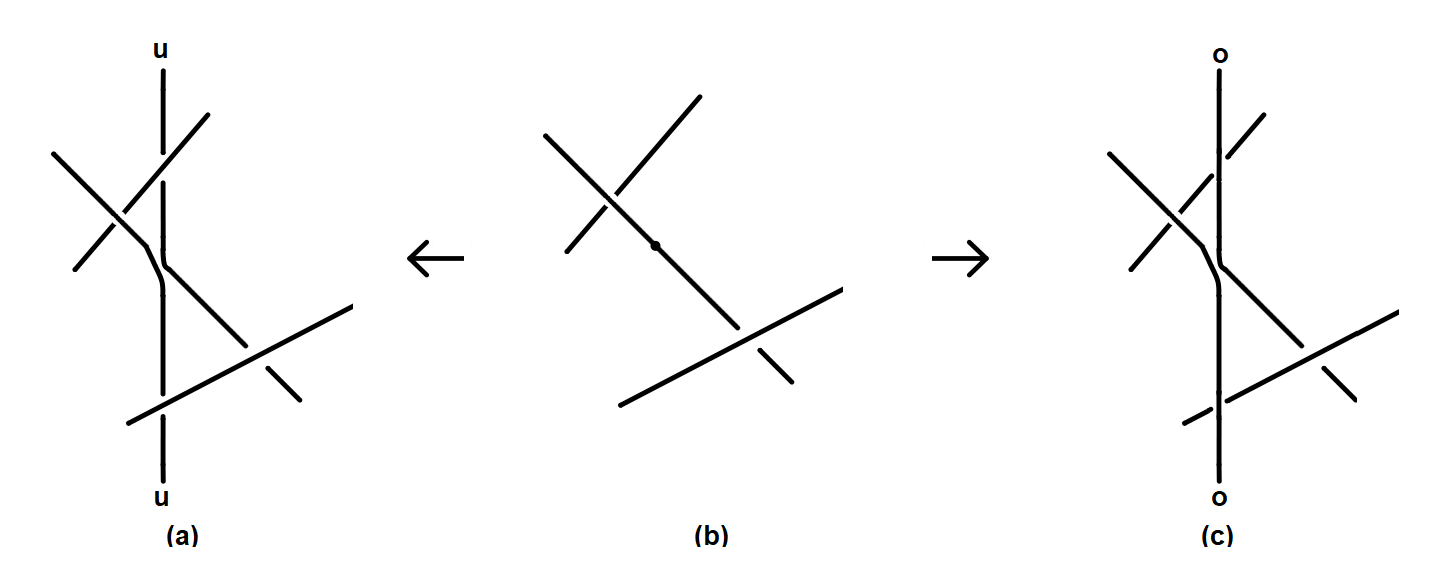}
        \caption{An $L_{under}$- and an $L_{over}$-move at the same point of a braidoid.}
    \label{fig:L-moves} 
\end{figure}

Clearly, the closures of two labelled braidoids that differ by an $L$-move are equivalent multi-knotoids. 
As an example, Fig.~\ref{fig:L1} illustrates two $L$-equivalent braidoid diagrams of the knotoid in Fig.~\ref{fig:knotoid}a. 

\smallbreak
In the theory, we further have the fake forbidden moves: A \textit{fake forbidden move} on a labeled braidoid diagram $b$ is a move that is forbidden  on $b$, but which upon closure induces a sequence of fake forbidden moves on the resulting multi-knotoid diagram. See  Fig.~\ref{fig:fakes}(b)-(c) for an example of a fake forbidden move   on a labeled braidoid diagram.

Finally, a {\it fake swing move} on a labeled braidoid diagram is a swing move which is not restricted, in the sense that the endpoint surpasses the vertical line of a pair of corresponding ends, but in the closure it is undone by a sequence of swing and fake forbidden moves on the resulting multi-knotoid diagram. See Fig.~\ref{fig:fakes}(a)-(b) for an example of a fake swing move. 

\begin{figure}[htp]
    \centering
    \includegraphics[width=11cm]{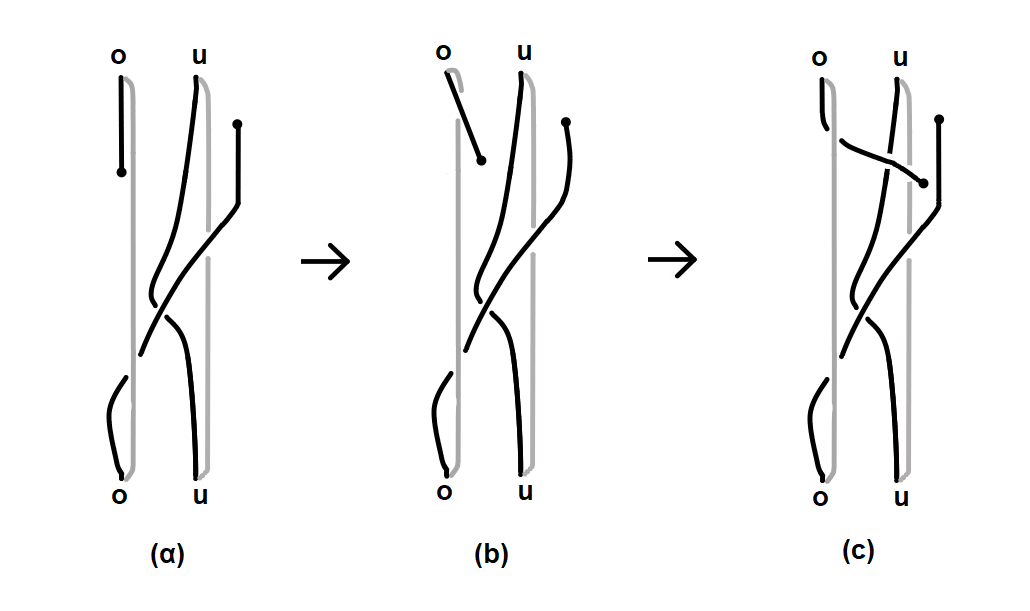}
        \caption{Fake swing moves and fake forbidden moves.}
    \label{fig:fakes}
\end{figure}

\begin{definition}[$L$-equivalence]\label{Def:L-equivalence}
    The $L$-moves together with labeled braidoid isotopy moves and fake swing moves generate an equivalence relation on labeled braidoid diagrams that is called the \textit{$L$-equivalence}, and is denoted by $\sim_{L}$.
\end{definition} 

The braidoiding algorithm induces a well-defined map from the set of equivalence classes of multi-knotoid diagrams in $\mathbb{R}^2$ to the set of $L$-equivalence classes of braidoid diagrams. Moreover, this map is the left and right inverse of the closure map from the set of $L$-equivalence classes of labeled braidoid diagrams to the set of equivalence classes of multi-knotoid diagrams, which is also a well-defined map. The above are based on the following lemmas, cf. \cite{gugumcu2021braidoids}:

\begin{lemma}
\label{br1}
    Let $K$ be a (multi)-knotoid diagram in $\mathbb{R}^2$ with a subdivision that satisfies the triangle conditions. Adding a subdividing point to an up-arc of $K$ and labeling the new up-arcs same as the labeling of the initial up-arc yields an $L$-equivalent labeled braidoid diagram.
\end{lemma}

\begin{lemma}
\label{br2}
    Labeling a free up-arc either with $o$ or $u$ does not change the resulting labeled braidoid diagram up to  $L$-equivalence.
\end{lemma}

\begin{corollary}
    If we have an appropriate choice of relabelling the up-arcs of $K$ resulting from a further subdivision on $K$ then, by Lemmas~\ref{br1} and~\ref{br2}, the resulting labeled braidoid diagrams are $L$-equivalent.
\end{corollary}

\begin{lemma}
    Two labeled braidoid diagrams that are obtained with respect to any two subdivisions $S_1$, $S_2$ on a multi-knotoid diagram $K$, satisfying the triangle conditions, are $L$-equivalent.
\end{lemma}

\begin{lemma}
    Two multi-knotoid diagrams in $\mathbb{R}^2$ are related to each other by planar isotopy and Reidemeister moves away from endpoints only if the corresponding labeled braidoid diagrams are related to each other by labeled braidoid isotopy moves, $L$-moves, fake swing moves or fake forbidden moves.
\end{lemma}

\noindent We can now state  the following theorem due to Gügümcü and Lambropoulou \cite{gugumcu2021braidoids}:

\begin{theorem}
\label{Nesli}
    The closures of two labeled braidoid diagrams are equivalent multi-knotoid diagrams in $\mathbb{R}^2$ if and only if the labeled braidoid diagrams are related to each other via $L$-equivalence moves.
\end{theorem}

The aim of this work is to modify Theorem~\ref{Nesli} so that the new equivalence classes of labeled braidoids will be in bijection with spherical multi-knotoids. For this we shall employ Theorem~\ref{thm:sphericalstate}, so as to work, equivalently, in the planar category.

\section{Spherical knotoids and  braidoids in spherical state}\label{sec:sphericalknotoidsandbraidoidsinsphericalstate}

We recall from Subsection~\ref{subsec:planarsphericalknotoids} that there is a  bijection between the set $\mathcal{K}(S^2)$ of spherical knotoids and the  equivalence classes of planar knotoids in spherical state  
(Theorem~\ref{thm:sphericalstate}).  We shall now introduce a similar concept in the set of braidoids, and we shall prove that a normal knotoid can be represented unambiguously by the closure of a braidoid in spherical state.

\subsection{The spherical state of a braidoid diagram}

\begin{definition}\label{def:sphericalstate}
    Let $b \in \mathbb{R} \times [0,1]$ be a (labeled) braidoid diagram and let $b_0$ be the 4-valent graph obtained from $b$ by forgetting the over/under data on the crossings. The space  $\mathbb{R} \times [0,1] \setminus b_0$ has at least one unbounded region. We say that the braidoid diagram $b$ is in \textit{canonical spherical state} if the leg of $b$ lies in the interior of the left unbounded region of $\mathbb{R} \times [0,1] \setminus b_0$ and it does not lie between two consecutive braidoid ends (i.e. it does not lie within the interior of the vertical strip determined by the neighbouring vertical lines passing through two consecutive pairs of corresponding braidoid ends). Moreover, we say that the braidoid diagram $b$ is in \textit{spherical state} if it is $L$-equivalent (recall Definition~\ref{Def:L-equivalence}) to a braidoid diagram in canonical spherical state. Finally, we shall refer to the leg of a braidoid diagram $b \in \mathbb{R} \times [0,1]$ that is not in spherical state as \textit{bounded}.
\end{definition} 

An example of non-isotopic diagrams in terms of restricted isotopy, one in spherical state and one not in spherical state, is shown in Fig.~\ref{fig:sphernonspher}, where the closures of the braidoid diagrams generate non equivalent planar or spherical knotoid diagrams. To see this, one can connect the endpoints of the diagrams with an arc passing under all other arcs, to recover two knot diagrams. It is easy to see that the knot diagram generated by Fig.~\ref{fig:sphernonspher}(a) represents the trefoil while the knot diagram generated by Fig.~\ref{fig:sphernonspher}(b) represents a figure-$8$ knot.

\begin{remark}\label{rem:examplenonspher}
    The main idea behind the spherical state of a braidoid diagram is that when closed, it generates a planar multi-knotoid diagram in spherical state, which in turn corresponds to a normal multi-knotoid diagram in $S^2$. This causes the need to extend the notion of canonical spherical state to spherical state. Indeed, the situation of canonical spherical state in Definition~\ref{def:sphericalstate} is not the only one that upon closure generate normal multi-knotoid diagrams. Consider, for example, the case where the leg of a (labeled) braidoid lies between two consecutive ends $i,i+1$ but its $y$-coordinate is greater (or lesser) than that of any point in an arc that crosses transversely the strip between the endpoints $i,i+1$. The closure of such a diagram is also a normal multi-knotoid diagram in $S^2$. In this case, such a braidoid diagram is $L$-equivalent to a braidoid diagram in spherical state, since by a finite sequence of fake swing and fake forbidden moves one can move the leg to a  position of Definition~\ref{def:sphericalstate}. The reader is referred to Fig.~\ref{fig:examplenonspherical}. Furthermore, if the leg lies in the interior of the right unbounded region and not in the vertical strip of two consecutive endpoints, then it can be transferred to the left by a sequence of fake swing and fake forbidden moves, which in turn correspond to $L$-equivalent braidoid diagrams. This case is illustrated abstractly in Fig.~\ref{fig:remark}.
\end{remark} 

\begin{figure}[htp]
    \centering
    \includegraphics[width=6.5cm]{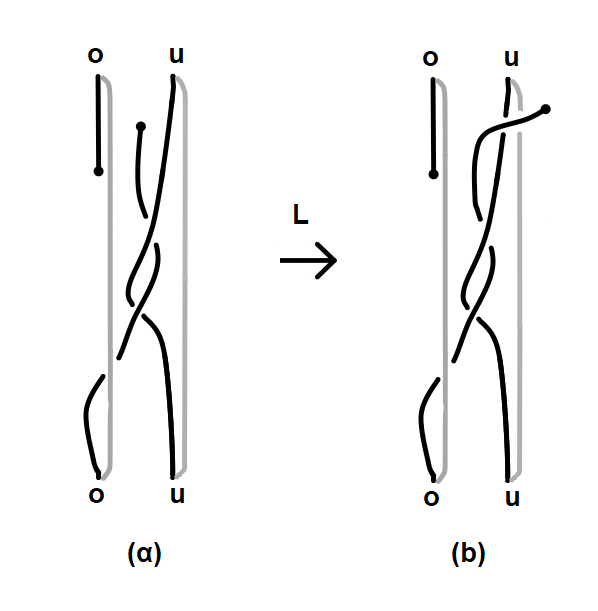}
    \caption{An example illustrating the first case of Remark~\ref{rem:examplenonspher}.}
    \label{fig:examplenonspherical}
\end{figure}

\begin{figure}[htp]
    \centering
    \includegraphics[width=9cm]{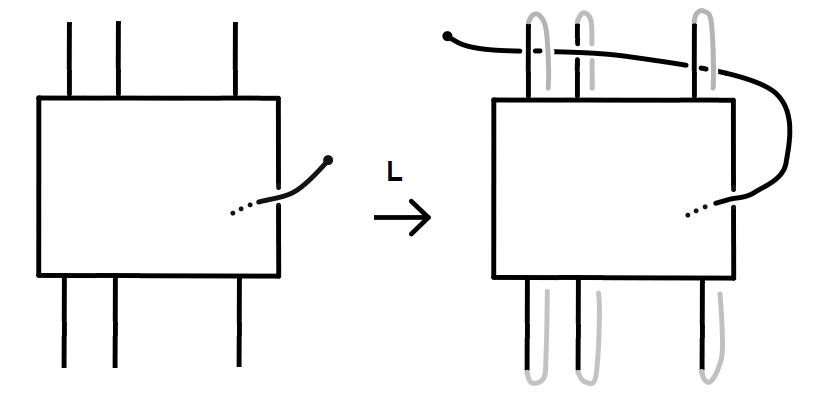}
    \caption{An example illustrating the second case of Remark~\ref{rem:examplenonspher}.}
    \label{fig:remark}
\end{figure}

\subsection{Braidoiding normal knotoid diagrams} 

   By Proposition~\ref{normal correspondence} every knotoid in $S^2$ has a normal representative, which corresponds bijectively to a planar knotoid in spherical state (Observation 6 and Definition~\ref{def:sphericalstate}), while the bijection between the set $\mathcal{K}(S^2)$ of spherical knotoids and the  equivalence classes of planar knotoids in spherical state  culminates in Theorem~\ref{thm:sphericalstate}.   Then the braidoiding algorithm described in Subsection~\ref{subsec:braidoiding} and planar isotopy lead to the following braidoiding result:
\begin{theorem} \label{thm:normalbraidoiding}
   Let $N$ be a normal multi-knotoid diagram in $S^2$ and $s(N)$ its image in $\mathbb{R}^2$ under the stereographic projection, which is a planar knotoid in spherical state. Then, $s(N)$ can be  deformed to the closure of a labeled braidoid diagram $b$ in canonical spherical state, which lies in the  planar equivalence class of $s(N)$.    
\end{theorem}

\begin{proof}
     Since $N$ is normal, then $s(N)$ is a planar diagram in spherical state.  We can assume $s(N)$ to be piecewise linear. By performing a planar isotopy, if necessary, the leg $l$ of $s(N)$ can be placed to the left  (or right) side of the diagram. Furthermore, by local isotopies, the edge of $s(N)$ directly attached to the leg can be assumed to be oriented downwards and sufficiently distant from the vertical lines to be used in the braidoiding algorithm. Applying, then, the  algorithm we obtain a labeled braidoid diagram $b$ whose leg lies in the left unbounded region of $\mathbb{R} \times [0,1] \setminus b_0$ (recall Definition~\ref{def:sphericalstate}) and does not lie in a vertical strip defined by consecutive braidoid ends, that is, a labeled braidoid diagram in canonical spherical state. See Fig.~\ref{fig:spherbraidoiding} for an abstraction of the steps described above. 
\end{proof}

\begin{figure}[htp]
    \centering
    \includegraphics[width=10cm]{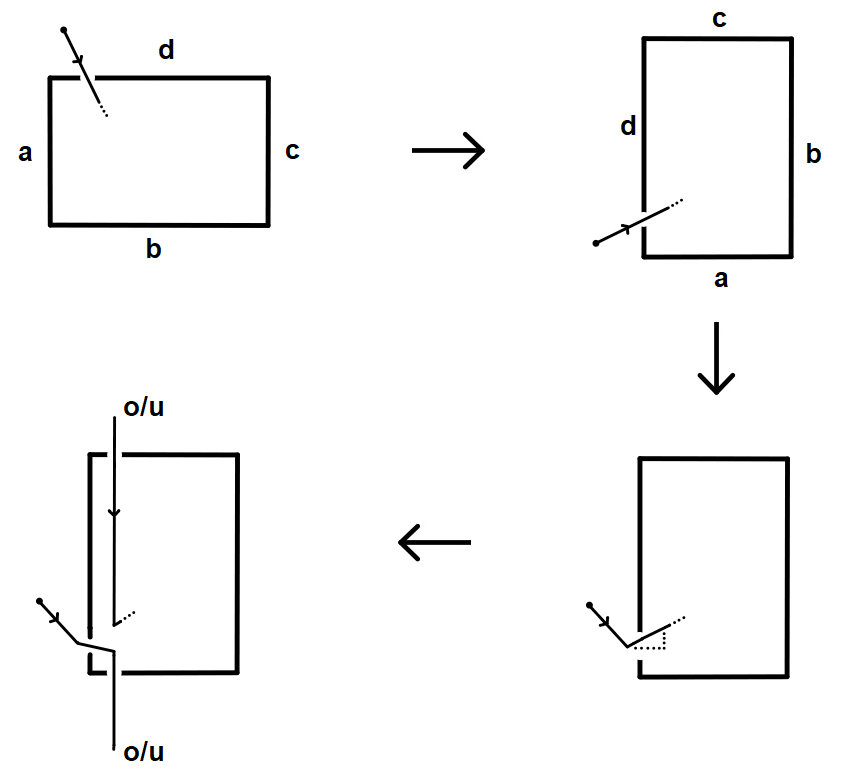}
        \caption{Preserving the leg in the unbounded region.}
    \label{fig:spherbraidoiding}
\end{figure}

\section{Spherical equivalence for labeled braidoids}\label{sec:Sphericalequivalenceforlabeledbraidoids}

 For the $L$-equivalence in the category of spherical multi-knotoids we need to add to the equivalence in Theorem~\ref{Nesli} extra diagrammatic moves that describe the effect of a spherical move between  multi-knotoid diagrams at the level of braidoid diagrams. 
 
\subsection{The spherical moves on braidoids}

We describe two types of moves that, along with the $L$-equivalence, shall formulate an equivalence relation in the set of labeled braidoid diagrams, such that the equivalence classes of labeled braidoid diagrams will be in bijection with equivalence classes of multi-knotoid diagrams in $S^2$. First we introduce some notation to help better describe our definitions. Let $p$ be a point in $\mathbb{R}^2$. By $x(p)$ we mean the projection of $p$ on the $x$-axis. Similarly $y(p)$ is the projection of $p$ on the $y$-axis.

\begin{definition}[sL I-move]
\label{slI}
    Let $b \in \mathbb{R} \times [0,1]$ be a labeled braidoid diagram not in spherical state. Let $l$ be the bounded leg of $b$  and $A$ be the strand of $b$ lying closest to $l$ (by closest we mean that the horizontal line segment $(-\infty,l]$ meets $A$ first from all other strands of $b$) that prevents $l$ from moving further toward the left unbounded region in the complement of $b_0$ (see Fig. \ref{fig:sphericalbraidoidmove}). Suppose also that $x(h)>x(l)$. Since we take $b$ to be in general position, $l$ can be assumed to  lie between the vertical lines of two consecutive ends $i$ and $i+1$. Pick now points $p_1$ and $p_2$ on $A$ so that the sub-arc $[p_1,p_2]$ of $A$ also lies entirely between the vertical lines of $i$ and $i+1$  and so that $y(p_2)<y(l)<y(p_1)$ and $x(l)<x(p_1)<x(p_2)$. The above conditions can always be satisfied by small planar isotopies.\\
    \indent Let $l_1$ and $l_2$ be the vertical lines that pass through $p_1$ and $p_2$ and meet $\mathbb{R} \times \{1 \} \cup \mathbb{R} \times \{0 \}$ at $\{x_t,x_b \}$ and $\{y_t,y_b \}$ respectively. Delete, next, the segment of $A$ between $p_1$ and $p_2$ and connect $p_2$ with $y_t$ with an arc that follows $l_2$ $ε$-close and passes under (or over) all other strands of $b$. Then, label the braidoid end at $y_t$ as $u$ (or $o$). Secondly, we connect $p_1$ with $x_b$ with an arc that follows $l_1$ $ε$-close and passes under (or over) all other strands of $b$ and label the braidoid end $x_b$ as $u$ (or $o$). Note that the newly formed strands are not corresponding.\\
    \indent Finally we connect $x_t$ with $y_b$ with an arc that consists of three sub-arcs $[x_t,d_1],[d_1,d_2]$ and $[d_2,y_b]$. The sub-arc $[d_1,d_2]$ lies entirely to the right of the diagram while the sub-arcs $[x_t,d_1]$ and $[d_2,y_b]$ lie on the upper and lower parts of the diagram respectively and pass through the strands respecting the labeling (i.e. passing over/under $u$/$o$-labeled ends as in Fig \ref{fig:sphericalbraidoidmove}.\\ 
    \indent Should $x(h)<x(l)$ we would reverse the direction of the move toward the right unbounded region. In the end, we  can simply move $l$ to the left region as in the second case of Remark \ref{rem:examplenonspher}. 
\end{definition}
 \noindent The reader is referred to Fig.~\ref{fig:close sl1} for the closure after an sL I-move.
 
\begin{figure}[htp]
    \centering
    \includegraphics[width=13cm]{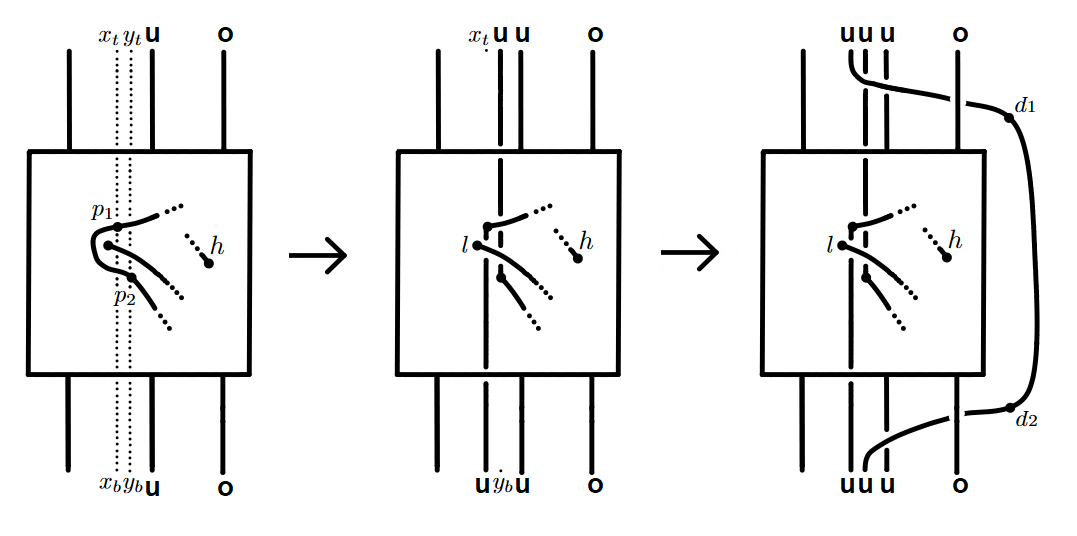}
        \caption{The sL I-move, assuming the head on the right-hand side.}
    \label{fig:sphericalbraidoidmove}
\end{figure}

 The second move is when the leg of a braidoid diagram $b$ lies in the vertical strip  between the $i$ and $i+1$ braidoid ends.

\begin{definition}[sL II-move] \label{slII}
    Let $b \in \mathbb{R} \times [0,1]$ be a labeled braidoid diagram with $n$ number of ends, not in spherical state so that the leg  of $b$ lies in the vertical strip between the $i$ and $i+1$ ends. Suppose again that $x(l)<x(h)$. Cut two small neighbourhoods of the upper and lower $i+1$ end, disjoint from the rest of the diagram, eliminating the $i$ end position and introduce a new end position to the right of the diagram. This will result in a new enumeration of the end positions where our newly introduced position will be $n$ and the braidoid ends $i+1$ to $n$ of the original diagram will shift from $i$ to $n-1$, however we do this without moving the braidoid diagram at all. We now join the new $n$ upper/lower end with the upper/lower end of the cut strand by two monotonically decreasing arcs that pass over or under the other strands respecting the labels (see Fig.~\ref{fig:sphericalbraidoidmove2}). The new end at position $n$ of the diagram will then be labeled either $o$ or $u$ according to the label of the original $i$ end. \\ 
    \indent Should $x(h)<x(l)$ we would reverse the direction of the move toward the left unbounded region. 
\end{definition}

\begin{figure}[htp] 
    \centering 
    \includegraphics[width=9cm]{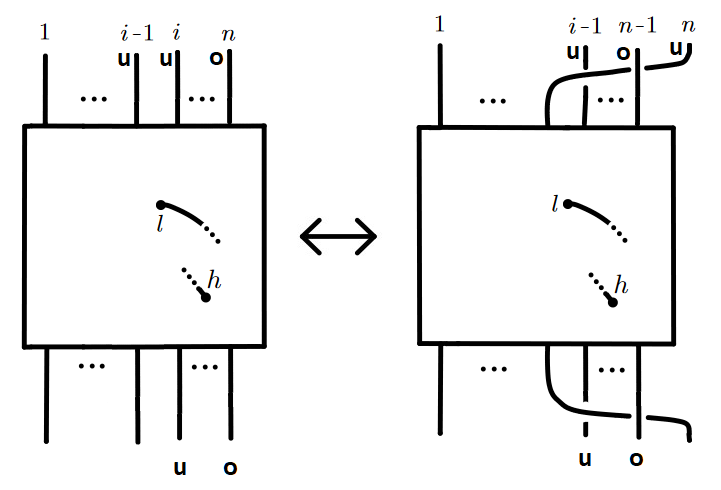} 
        \caption{The sL II-move, assuming the head on the right-hand side.} 
    \label{fig:sphericalbraidoidmove2} 
\end{figure} 

\begin{remark} 
Note that while in Definitions~\ref{slI} and~\ref{slII} we free the leg to be in accordance with \cite{turaev2012vladimir}, these operations can be done for any endpoint. 
\end{remark} 

We shall now give examples to better visualize the effect of the $sL$-moves.  Fig.~\ref{fig:SL1} shows an example of an sL I-move. One can see that the closures of the braidoid diagrams in Fig.~\ref{fig:L1} (as it was also the case for the ones in Fig.~\ref{fig:SL1}) represent the knotoid in Fig.~\ref{fig:knotoid}a. This knotoid  is trivial in $S^2$ whereas in $\mathbb{R}^2$ it is not, as shown in \cite{turaev2012vladimir}. To get from \ref{fig:L1}b to \ref{fig:SL1}c we perform two sL II-moves in the third and fourth end of diagram \ref{fig:L1}b. Then, by performing an sL I-move we can move the endpoint further to the left.  Further, the diagram in Fig. \ref{fig:L1}a is not $L$-equivalent with Fig. \ref{fig:SL1}c (or Fig. \ref{fig:SL1}d).

\begin{figure}[htp]
    \centering
    \includegraphics[width=9cm]{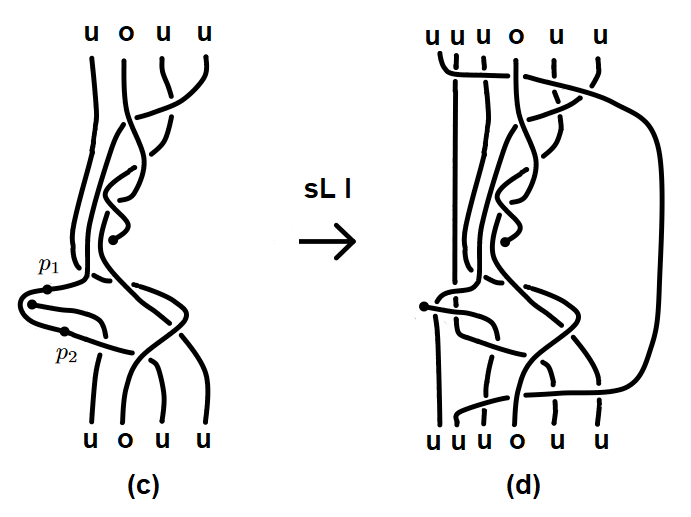}
        \caption{Example of the sL I-move.}
    \label{fig:SL1}
\end{figure}

 Fig.~\ref{fig:ex1} illustrates another example of an sL II-move. Note that the closure of the diagram in Fig.~\ref{fig:ex1}a is the knotoid in Fig.~\ref{fig:knotoid}b. In $S^2$ this knotoid is trivial while in $\mathbb{R}^2$ it is not. From (a) to (b) there is an  sL II-move performed. The closure of the diagram in Fig.~\ref{fig:ex1}b or Fig.~\ref{fig:ex1}c is the trivial knotoid. This means that the diagrams in Figs.~\ref{fig:ex1}a and~\ref{fig:ex1}c are not $L$-equivalent. 

\begin{figure}[htp]
    \centering
    \includegraphics[width=9cm]{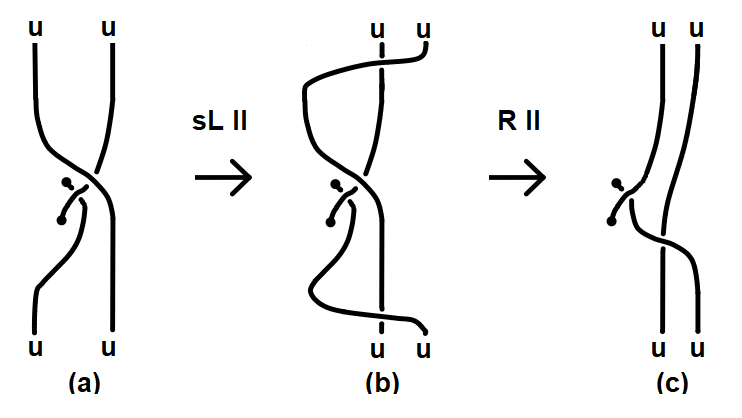}
        \caption{Example of the sL II-move.}
    \label{fig:ex1}
\end{figure}

\subsection{Braidoid equivalence for spherical knotoids}

\begin{definition}
    Two labelled braidoid diagrams $b_1,b_2$, not necessarily with the same number of strands, are said to be \textit{sL-equivalent} if they are related to each other by a finite number of $L$-equivalence moves and the sL I and sL II-moves.
\end{definition}

\noindent We shall now state a key lemma on which our main result is based.
\begin{lemma}
\label{spherical}
    Any labeled braidoid diagram $b$ is $sL$-equivalent to a diagram $b^*$ that is in canonical spherical state.  
\end{lemma}

\begin{proof}
    Suppose that $b$ has $n$ ends and $l$ lies between $i-1$ and $i$, and that $x(l)<x(h)$. Then all closing arcs on the vertical lines corresponding to ends $i-1,i-2, \ldots, 1$ will obstruct $l$ from moving to the left unbounded region of $\mathbb{R} \times [0,1] \setminus b_0$. By performing consecutive sL II-moves, starting from position $i-1$, we can successively carry all end positions $i-1, \ldots, 1$ to the right side of the diagram (recall Fig.~\ref{fig:sphericalbraidoidmove2}), thus obtaining a new labeled braidoid diagram $b'$, $sL$-equivalent to $b$, in which $l$ lies to the left of the vertical line passing through the upper and lower $1$-ends. If the horizontal line $(-\infty,l]$ meets no strands of $b'$ then $b'$ is in canonical spherical state. If not, then let $a_1, \ldots, a_k$ be the intersection points of $b'$ with $(-\infty,l]$, with $x(a_k)<\ldots<x(a_1)$. We can assume that $a_1,\ldots, a_k$ are not crossing points and let $A_1,\ldots ,A_k$ be a choice of arcs containing $a_1,\ldots ,a_k$ respectively, for performing sL I-moves. Performing the move to $A_1$ we obtain a labeled braidoid diagram $b_1'$ with two new ends, whose closure is equivalent to the closure of $b$ in $S^2$ and the line $(-\infty,l]$ meets the rest of $b_1'$ in $k-1$ points. Proceeding inductively, we obtain a labeled braidoid diagram $b_k'$ where $(-\infty,l]$ has no intersection points with $b_k'$. This means that, by Definition~\ref{def:sphericalstate},  $b_k'$ is in canonical spherical state.
\end{proof}

\begin{theorem}
\label{Kokki}
    The closures of two labeled braidoid diagrams are equivalent multi-knotoids in $S^2$ if and only if the labeled braidoid diagrams are $sL$-equivalent.
\end{theorem}

\begin{proof}
    For the `if' part, note first that $L$-moves give equivalent multi-knotoids in $\mathbb{R}^2$ and by extension in $S^2$. Further, one can easily see that the closures of two labeled braidoid diagrams that differ by an sL I-move or an sl II-move give rise to equivalent multi-knotoids in $S^2$. Indeed, for the sL I-move see Fig. \ref{fig:close sl1} and observe how the diagrams $(b)$ and $(c)$ differ by a spherical move in $S^2$.
    
\begin{figure}[htp]
    \centering
    \includegraphics[width=13cm]{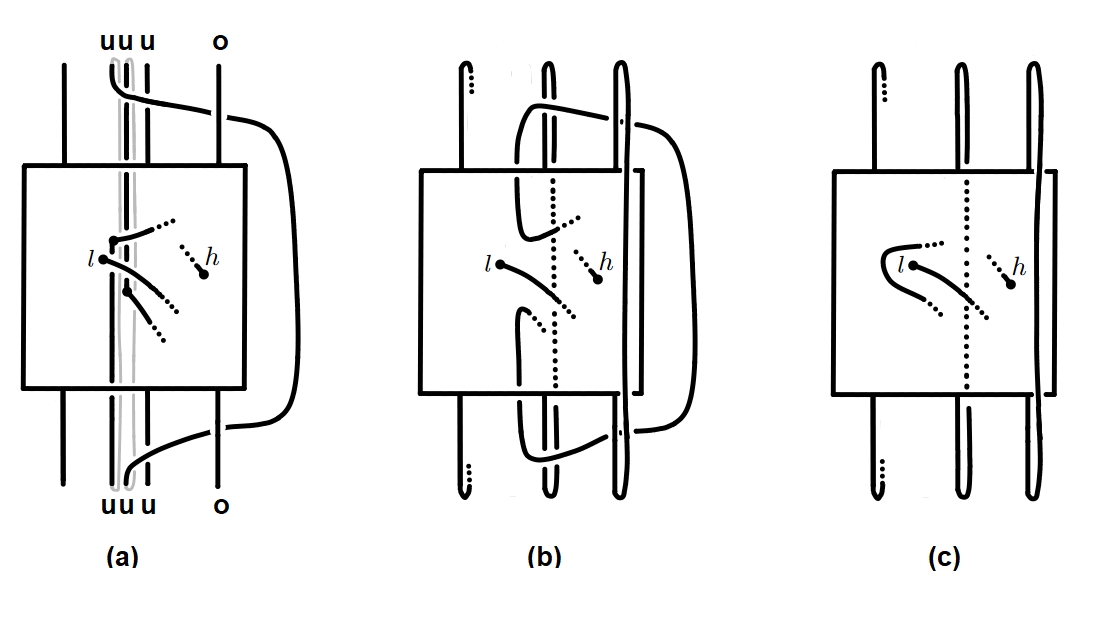}
        \caption{The closure after an sL I-move is performed.}
    \label{fig:close sl1}
\end{figure}

For the sl II-move see Fig.~\ref{fig:close sl2} and observe how one can move the closure arc in the right diagram by a sequence of moves consisting of a spherical move, Reidemeister II and III moves and two opposite Reidemeister I moves and bring it to its original position on the left diagram).

\begin{figure}[htp]
    \centering
    \includegraphics[width=10cm]{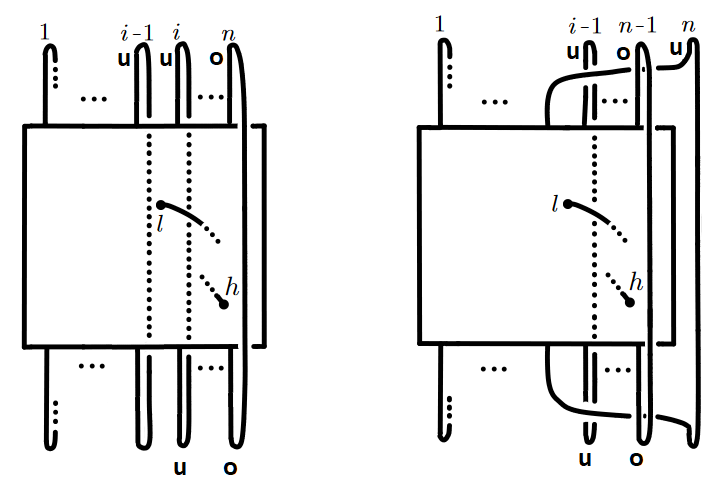}
        \caption{The closure after an sL II-move is performed.}
    \label{fig:close sl2}
\end{figure}

    Note that the reason why spherical moves can be performed in the braidoid closures after the $sL$-moves is because the $sL$-moves respect the labeling of the braidoid diagram. Should any of the arcs of the $sL$-moves that pass through the upper/lower strands fail to respect the labeling (i.e. passing over an $o$-labeled strand or under a $u$-labeled strand, then upon closure, these arcs will become trapped by the closure arcs and fail to perform the necessary spherical moves. 

    For the `only if' part, let $b_1, b_2$ be labeled braidoid diagrams (not necessarily with the same number of ends) that have equivalent closures in $S^2$. By Lemma~\ref{spherical} there exist labeled braidoids $b_1^{*}$, $b_2^{*}$ in canonical spherical state, so that $b_1 \sim_{sL} b_1^{*}, \ b_2 \sim _{sL} b_2^{*}$ and from the above:
    $$\hat{b_1^{*}} \simeq _{S^2} \hat{b_1} \simeq _{S^2} \hat{b_2} \simeq _{S^2} \hat{b_2^{*}}$$ 
    Now $\hat{b_1^{*}}$ and $\hat{b_1^{*}}$ are normal multi-knotoid diagrams equivalent in $S^2$. By Turaev \cite{turaev2012vladimir} (recall Proposition~\ref{normal correspondence}) such multi-knotoid diagrams are equivalent in $\mathbb{R}^2$ and therefore by \cite{gugumcu2021braidoids} we have that $$b_1^{*} \sim_L b_2^{*}$$ 
    Hence $$b_1 \sim _{sL} b_1^{*} \sim _{L} b_2^{*} \sim _{sL} b_2$$ This means that $b_1 \sim_{sL} b_2$ and the proof of the Theorem is complete.    
\end{proof}

\begin{remark}
    The spherical move in Fig.~\ref{fig:sphericalmove} can be described on a braidoid level in the following way: Suppose that we have a knotoid diagram in $S^2$ and we want to perform a spherical move on an arc of the diagram that can be seen as outmost. Up to knotoid diagram equivalence, the arc in question can be an up-arc or a down-arc. In the case where the arc is an up-arc, we can perform the braidoiding algorithm on the rest of the diagram leaving our arc for last. This means that the braoidoid diagram whose closure represents the original knotoid diagram has the up-arc, where the spherical move takes place, as a closure arc for the last (or first) braidoid end. Then, the spherical move can be interpreted as an sL II-move on the closure arc of the last braidoid end (see Fig.~\ref{fig:spher_rem_1}).\\
    \indent Closing the end braidoid diagram we recover the original knotoid diagram with a spherical move performed on the outmost up-arc. For the other case, where the spherical move is performed on a down-arc, by a planar isotopy, we can bring the arc in a position so that the up-arc parts of the arc are oriented downwards and notice that, when the braidoiding algorithm is performed, this particular arc will be ignored since it is a down-arc. So we can present the knotoid diagram as the closure of a braidoid diagram that has the down-arc as part of some braidoid strand. Then the spherical move can be presented as the result of an sL I-move on the points $p_1$ and $p_2$, as can be seen in Fig.~\ref{fig:spher_rem_2}.

\begin{figure}[htp]
    \centering
    \includegraphics[width=11cm]{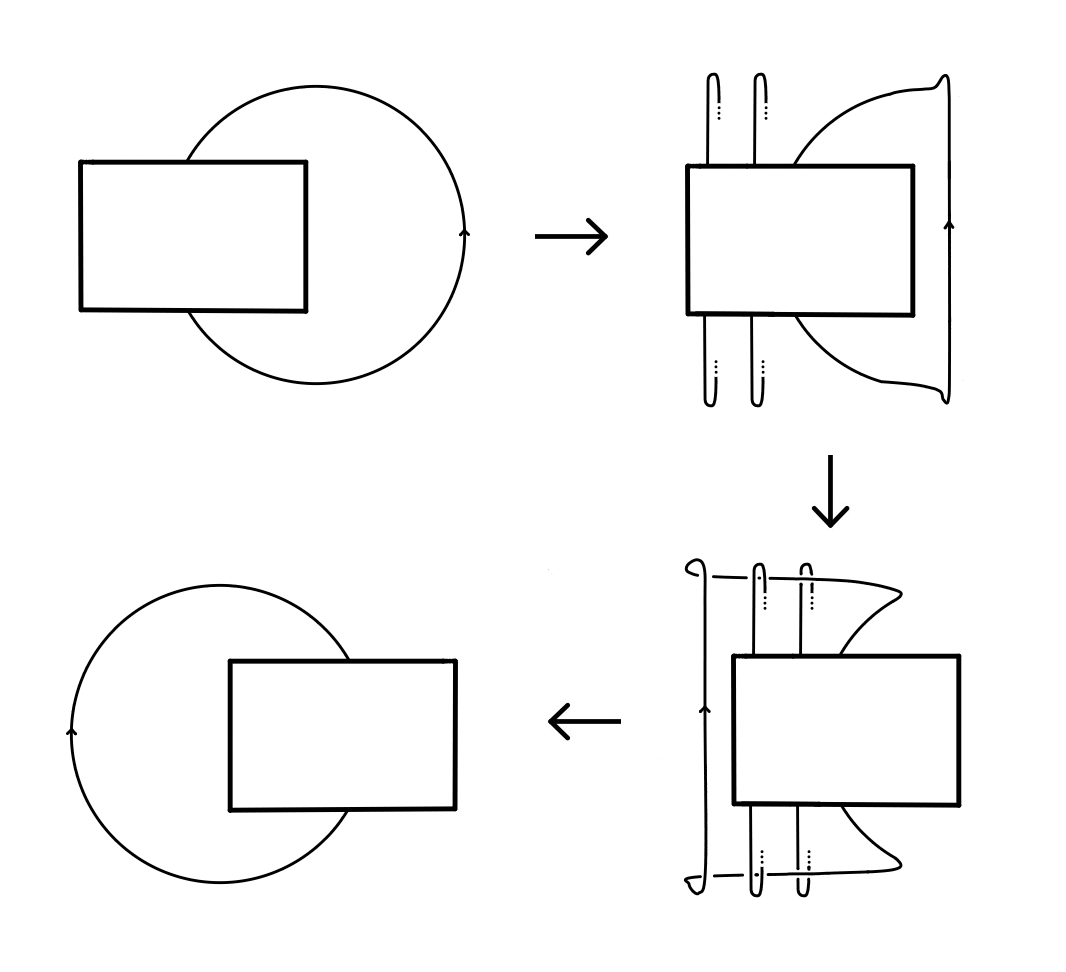}
        \caption{The spherical move on an up-arc through an sL II-move.}
    \label{fig:spher_rem_1}
\end{figure}

\begin{figure}[htp]
    \centering
    \includegraphics[width=16cm]{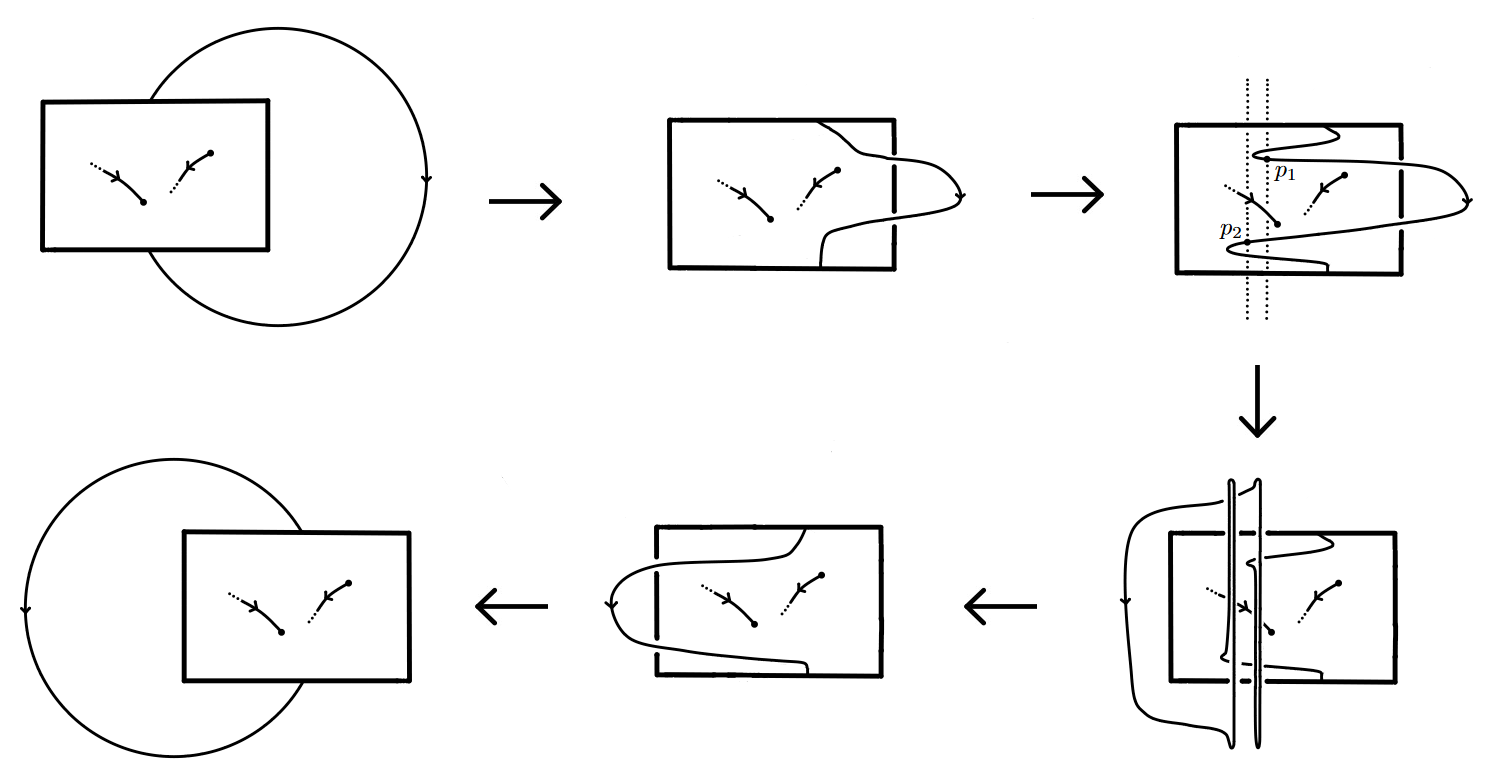}
        \caption{The spherical move on an down-arc through an sL I-move.}
    \label{fig:spher_rem_2}
\end{figure}

\end{remark}

\noindent Using Theorems \ref{Tur1} and \ref{Kokki} we also have the following:

\begin{corollary}
    There is a bijection between the set of $sL$-equivalence classes of labeled braidoid diagrams and label preserving ambient isotopy classes of simple multi-theta curves.
\end{corollary}

\section{Discussion}

The theory of braidoids is a new diagrammatic setting for knotoids extending at the same time the theory of braids. The Markov theorem provides an algebraic viewpoint for links in $S^3$, giving rise to a multitude of link and 3-manifold invariants constructed by way of braid representations. Braidoids do not possess an obvious algebraic structure as of yet, however work has been done toward this goal. Nevertheless, the  $L$-moves provide the framework for formulating (planar and spherical) braidoid equivalences.   Our hope is that the present paper provides a useful tool in understanding and expanding this already rich and active area of mathematical research.

\nocite{*}
\bibliographystyle{plain}
\bibliography{refs}

\begin{thebibliography}{1}

\bibitem{dabrowski2019theta}
Pawel Dabrowski-Tumanski, Dimos Goundaroulis, Andrzej Stasiak, and Joanna~I
  Sulkowska.
\newblock $θ$-curves in proteins.
\newblock {\em arXiv e-prints, (accepted in Protein Science (2024))}, pages
  arXiv--1908, 2019.

\bibitem{goundaroulis2017studies}
Dimos Goundaroulis, Julien Dorier, Fabrizio Benedetti, and Andrzej Stasiak.
\newblock Studies of global and local entanglements of individual protein
  chains using the concept of knotoids.
\newblock {\em Scientific reports}, 7(1):6309, 2017.

\bibitem{goundaroulis2017topological}
Dimos Goundaroulis, Neslihan G{\"u}g{\"u}mc{\"u}, Sofia Lambropoulou, Julien
  Dorier, Andrzej Stasiak, and Louis Kauffman.
\newblock Topological models for open-knotted protein chains using the concepts
  of knotoids and bonded knotoids.
\newblock {\em Polymers}, 9(9):444, 2017.

\bibitem{gugumcu2017new}
Neslihan G{\"u}g{\"u}mc{\"u} and Louis~H Kauffman.
\newblock New invariants of knotoids.
\newblock {\em European Journal of Combinatorics}, 65:186--229, 2017.

\bibitem{gugumcu2017knotoids}
Neslihan G{\"u}g{\"u}mc{\"u} and Sofia Lambropoulou.
\newblock Knotoids, braidoids and applications.
\newblock {\em Symmetry}, 9(12):315, 2017.

\bibitem{gugumcu2021braidoids}
Neslihan G{\"u}g{\"u}mc{\"u} and Sofia Lambropoulou.
\newblock Braidoids.
\newblock {\em Israel Journal of Mathematics}, 242:955--995, 2021.

\bibitem{lambropoulou1997markov}
Sofia Lambropoulou and Colin~P Rourke.
\newblock Markov's theorem in 3-manifolds.
\newblock {\em Topology and its Applications}, 78(1-2):95--122, 1997.

\bibitem{turaev2012vladimir}
V~Turaev.
\newblock Knotoids.
\newblock {\em Osaka J. Math}, 49:195--223, 2012.

\end{thebibliography}

\end{document}